\documentclass[a4paper,11pt]{amsart}

\usepackage{amssymb}
\usepackage{amsmath}
\usepackage{amsthm}
\usepackage{amscd}
\usepackage{verbatim}
\usepackage{graphicx}
\usepackage[all]{xy}
\usepackage{url}

\usepackage{mathtools}
\usepackage{dsfont}

\usepackage{comment}
\usepackage{color}
\usepackage{url}

\DeclareFontEncoding{OT2}{}{} % to enable usage of cyrillic fonts
\DeclareTextFontCommand{\textcyr}{\fontencoding{OT2}\fontfamily{wncyr}\fontseries{m}\fontshape{n}\selectfont}

\definecolor{brown}{RGB}{150,100,0}

\definecolor{purple}{RGB}{128,0,128}

\definecolor{grey}{RGB}{128,128,128}

\definecolor{gray}{RGB}{32,32,32}

\theoremstyle{plain}

\newtheorem{theorem}{Theorem}
\newtheorem{proposition}{Proposition}
\newtheorem{lemma}{Lemma}
\newtheorem{corollary}{Corollary}
\newtheorem*{corollary*}{Corollary}

\theoremstyle{definition}

\newtheorem{remark}{Remark}
\newtheorem{example}{Example}
\newtheorem*{example*}{Example}
\newtheorem{subsec}{}

\title[Real component group]
{A short proof of Timashev's theorem\\ on the real component group\\  of a real reductive  group}

\author{Mikhail Borovoi and Ofer Gabber}

\address{Borovoi: Raymond and Beverly Sackler School of Mathematical Sciences,
Tel Aviv University, 6997801 Tel Aviv, Israel}
\email{borovoi@tauex.tau.ac.il}

\address{Gabber: IHES, Le Bois-Marie, 35, Route de Chartres,
F-91440 Bures-sur-Yvette, France}
\email{gabber@ihes.fr}



\newcommand{\into}{\hookrightarrow}
\newcommand{\onto}{\twoheadrightarrow}

\newcommand{\labelt}[1]{\xrightarrow{\makebox[1.2em]{\scriptsize ${#1}$}}}
\newcommand{\labelto}[1]{\xrightarrow{\makebox[1.5em]{\scriptsize ${#1}$}}}

\newcommand{\longisoto}{{\ \labelt{\raisebox{-1.ex}{$\sim$}}\ }}
\newcommand{\isoto}{\longisoto}

\newcommand{\hs}{\kern 0.8pt}
\newcommand{\hssh}{\kern 1.2pt}

\newcommand{\hm}{\kern -0.8pt}

\newcommand{\emm}{\bfseries}

\newcommand{\Cd}{\mathds{C}}

\newcommand{\Zd}{\mathds{Z}}
\newcommand{\Rd}{\mathds{R}}

\newcommand{\Gd}{\mathds{G}}

\newcommand{\Z}{{\Zd}}

\newcommand{\R}{{\Rd}}
\def\C{{\Cd}}
\newcommand{\G}{{\Gd}}

\newcommand{\X}{{{\sf X}}}

\def\Hom{{\rm Hom}}

\def\Gal{{\rm Gal}}

\def\coker{{\rm coker\,}}
\def\im{{\rm im\,}}

\newcommand{\upgam}{{\hs^\gamma\hm}}

\newcommand{\ssc}{{\rm sc}}

\newcommand{\Ga}{\Gamma}

\def\cc{\raise 1.7pt \hbox{\Tiny{$\bullet$}}}

\newcommand{\Ho}{{\mathrm{H}\kern 0.4pt}}
\newcommand{\Zl}{{\mathrm{Z}\kern 0.2pt}}
\newcommand{\Bd}{{\mathrm{B}\kern 0.2pt}}

\newcommand{\spl}{{\rm spl}}

\newcommand{\GmC}{\G_{{\rm m},\C}}
\newcommand{\GmR}{\G_{{\rm m},\R}}

\def\pia{{\pi_1^{\rm alg}}}

\newcommand{\pio}{\pi_0\hs}
\newcommand{\Hoz}{{\widehat \Ho{^0}}}

\begin{document}

%\date{\today}

\begin{abstract}
Using results of Cartan,  Matsumoto, and Casselman,
we give a short proof of Timashev's theorem
computing the real component group $\pio G(\R)$
of a connected reductive $\R$-group $G$
in terms of a maximal torus of $G$ containing a maximal split torus.
\end{abstract}

\begin{comment}
%Abstract for arXiv
Using results of Cartan,  Matsumoto, and Casselman,
we give a short proof of Timashev's theorem
computing the real component group \pi_0 G(R)
of a connected reductive real algebraic group G
in terms of a maximal torus of G containing a maximal split torus.
\end{comment}

\maketitle

Let $G$ be a (connected) reductive group over the field $\R$ of real numbers
(we follow the convention of SGA3, where reductive groups
are assumed to be connected).
The assumption that $G$ is connected means that
the group of $\C$-points $G(\C)$ is connected,
but the group of $\R$-points $G(\R)$ might not be connected.
We wish to compute the {\em real component group} $\pio G(\R)=G(\R)/G(\R)^0$,
where $G(\R)^0$ denotes the identity component of $G(\R)$.

\begin{example*}
Let $G=\GmR^d$\hs, the split $\R$-torus of dimension $d$.
Then
\[G(\R)=(\R^\times)^d,\quad\ \pio G(\R)\cong (\Z/2\Z)^d.\]
\end{example*}

\begin{theorem}[\'Elie Cartan]\label{t:Cartan}
Let $G$ be a {\emm simply connected semisimple} $\R$-group.
Then $G(\R)$ is connected, that is, $\pio G(\R)=\{1\}$.
\end{theorem}

\begin{proof}
See Borel and Tits  \cite[Corollary 4.7]{Borel-Tits-C},
or Gorbatsevich, Onishchik and Vinberg  \cite[Theorem 2.2 in Section 4.2.2]{GOV},
or Platonov and Rapinchuk \cite[Proposition 7.6 on page 407]{PR}.
Borel and Tits \cite[Section 4.8]{Borel-Tits-C} write
that this result goes back to Cartan.
\end{proof}

We denote $\Ga=\Gal(\C/\R)=\{1,\gamma\}$,
where $\gamma$ is complex conjugation.
Let $A$ be a $\Ga$-module, that is, an abelian group with a $\Ga$-action.
Recall that the 0-th Tate cohomology group is defined by
\[\Hoz A=\Zl^0\hm A\hs/\hs\Bd^0\hm A,\ \ \text{where}\ \
\Zl^0\hm A=A^\Ga\coloneqq\{a\in A \mid \hm\upgam a=a\},\ \Bd^0\hm A=\{a'+\hm\upgam a'\mid a'\in A\}.\]

Let $T$ be an $\R$-torus.
We consider the cocharacter group
\[\X_*(T)=\Hom(\GmC\hs, T_\C).\]

\begin{theorem}[{Casselman \cite[Section 5]{Casselman};
see also \cite[Corollary 3.10]{Borovoi-Timashev-bis}}]
\label{t:Casselman}
For an $\R$-torus $T$, the homomorphism
\[\X_*(T)\to T(\C),\quad \nu\mapsto \nu(-1)\ \,\text{for}\ \nu\in \X_*(T)\]
induces an isomorphism
\[\Hoz\,\X_*(T)\isoto\pio T(\R).\]
\end{theorem}

For absolutely simple  $\R$-groups $G$ {\em of adjoint type},
the finite abelian groups $\pio G(\R)$ were tabulated by
Matsumoto \cite[Appendix]{Matsumoto}; see also
Th\v{a}\'ng \cite{Thang} and   Adams and Ta\"{\i}bi \cite{AT}.
Until 2021, the only known result on $\pio G(\R)$
for a general reductive $\R$-group $G$
was Matsumoto's theorem:

\begin{theorem}[{\hs Matsumoto \cite[Corollary of Theorem 1.2]{Matsumoto},
see also Borel and Tits \cite[Theorem 14.4]{Borel-Tits}}\hs]
\label{t:Matsumoto}
Let $T_{\rm spl}\subseteq G$ be a maximal split $\R$-torus
in a reductive $\R$-group $G$.
Then the natural homomorphism
\[\pio T_{\rm spl}(\R)\to \pio G(\R)\]
is surjective.
\end{theorem}

\begin{corollary*}
$\pio G(\R)\simeq (\Z/2\Z)^n$, where
$n\le{\rm rank}_\R \hs G\coloneqq \dim T_{\rm spl}$.
\end{corollary*}

For a reductive $\R$-group $G$,
let $G^\ssc$ denote the universal cover
of the commutator subgroup $[G,G]$ of $G$,
which is simply connected;
see \cite[Proposition (2:24)(ii)]{Borel-Tits-C}
or \cite[Corollary A.4.11]{CGP}.
Consider the composite homomorphism
\[\rho\colon G^\ssc\onto[G,G]\into G,\]
which is in general is neither injective nor surjective.
For a maximal torus  $T\subseteq G$,
we denote
\[T^\ssc=\rho^{-1}(T)\subseteq G^\ssc.\]
We consider the {\em algebraic fundamental group}
of \cite{Borovoi-Memoir} defined by
\[\pia G=\X_*(T)/\rho_*\X_*(T^\ssc).\]
The Galois group $\Ga$ naturally acts on $\pia G$, and the $\Ga$-module $\pia G$
is well defined (does not depend on the choice of $T$
up to a transitive system of isomorphisms);
see \cite[Lemma 1.2]{Borovoi-Memoir}.

In \cite{Borovoi-Timashev-bis} Timashev and the first-named author
defined an action of the finite abelian group
$\Hoz \pia G$ on the first Galois cohomology set $\Ho^1(\R,G)$,
and proved that $\pio G(\R)$ is canonically isomorphic to
$\big(\Hoz \pia G\big) _{[1]}$, the stabilizer in $\Hoz \pia G$
of the neutral cohomology class
$[1]\in \Ho^1(\R,G)$; see \cite[Theorem 0.5]{Borovoi-Timashev-bis}.

Very recently Timashev proved the following result
computing $\pio G(\R)$ in terms of a maximal torus of $G$
containing a maximal split torus:

\begin{theorem}[{\hs Timashev \cite[Theorem 6]{Timashev}\hs}]
\label{t:Timashev}
Let $G$ be a reductive $\R$-group,
and $T\subseteq G$ be a maximal torus
{\emm containing a maximal split torus $T_\spl$ of $G$}.
Then there is a canonical isomorphism
\[ \X_*(T)^\Ga/\hs\big\{\nu+\hm\upgam \nu+\rho_*(\nu^\ssc)\ |\ \nu\in\X_*(T),
\ \nu^\ssc\in\X_*(T^\ssc)\big\} \,\isoto\, \pio G(\R)\]
sending the class of $\nu\in X(T)^\Ga$ to $\nu(-1)\cdot G(\R)^0\in\pio G(\R).$
\end{theorem}

In view of Casselman's theorem (Theorem \ref{t:Casselman}),  
Timashev's theorem can be stated as follows:

\begin{theorem}\label{t:Timashev-pi0}
For $G$ and $T$ as in Theorem \ref{t:Timashev},
the homomorphism $\pio T(\R)\to\pio G(\R)$ induced by the inclusion $T\into G$,
induces an isomorphism
\[\coker\big[\hs\pio T^\ssc(\R)\to\pio T(\R)\hs\big]\hs\isoto\hs \pio G(\R).\]
\end{theorem}

Below we give a one-page proof of Theorem \ref{t:Timashev-pi0}
using Cartan's theorem and Matsumoto's theorem.

\begin{proposition}\label{p:Gabber}
For any maximal torus $T$ in a reductive $\R$-group $G$,
consider the homomorphism
\[i_*\colon \pio T(\R)\to \pio G(\R)\]
induced by the inclusion homomorphism $i\colon T\into G$.
Then
\[\ker i_*=\im\big[\rho_*\colon\pio T^\ssc(\R)\to \pio T(\R)\big].\]
\end{proposition}

\begin{proof}
We have commutative diagrams
\[
\xymatrix{
T^\ssc\ar[r]^-{i^\ssc}\ar[d]_-\rho & G^\ssc\ar[d]^-\rho &&\pio T^\ssc(\R)\ar[r]^-{i^\ssc_*}\ar[d]_-{\rho_*}  & \pio G^\ssc(\R)\ar[d]^-{\rho_*} \\
T\ar[r]^-i  & G                                         &&\pio T(\R)\ar[r]^-{i_*}                            &\pio G(\R)
}
\]
where  $\pio G^\ssc(\R)=\{1\}$ by Cartan's theorem (Theorem \ref{t:Cartan}).
It follows that
\[\im\big[\rho_*\colon\pio T^\ssc(\R)\to \pio T(\R)\big]\,\subseteq\,\ker i_*.\]

Let $S=Z(G)^0$, the identity component of the center $Z(G)$
of our reductive group $G$.
We have a surjective $\R$-homomorphism
\[S\times G^\ssc\to G, \quad (s,g^\ssc)\mapsto s\cdot\rho(g^\ssc)
\ \ \text{for}\ s\in S(\C),\ g^\ssc\in G^\ssc(\C).\]
By the implicit function theorem, the induced homomorphism
\[ S(\R)^0\times G^\ssc(\R)^0\to G(\R)^0\]
is surjective.

Let $t\in T(\R)\subseteq G(\R)$, and assume that
$t\cdot T(\R)^0\in \ker i_*$\hs, that is, $t\in G(\R)^0$.
By the above, we can write
$t=s\cdot\rho(g^\ssc)$ where
$s\in S(\R)^0\subseteq T(\R)^0$ and $g^\ssc\in G^\ssc(\R)^0$.
Since $t\in T(\R)$ and $s\in T(\R)$, we see that $\rho(g^\ssc)\in T(\R)$,
whence $g^\ssc\in T^\ssc(\R)$.
We conclude that
\[ t\in T(\R)^0\cdot\rho\big(T^\ssc(\R)\big),\]
that is,
\[t\cdot T(\R)^0\in \im\big[\rho_*\colon\pio T^\ssc(\R)\to \pio T(\R)\big].\]
Thus
\[\ker i_*\,\subseteq\, \im\big[\rho_*\colon\pio T^\ssc(\R)\to \pio T(\R)\big],\]
which completes the proof of the proposition.
\end{proof}

\begin{proof}[Proof of Theorem \ref{t:Timashev-pi0}.]
Consider the sequence of groups and homomorphisms
\begin{equation}\label{e:rho-i}
\pio T^\ssc(\R)\labelto{\rho_*}\pio T(\R)\labelto{i_*}\pio G(\R)\to 0.\tag{$*$}
\end{equation}
By Proposition \ref{p:Gabber}, this sequence is exact at $\pio T(\R)$.
By Matsumoto's theorem (Theorem \ref{t:Matsumoto}),
the natural  homomorphism $\pio T_\spl(\R)\to\pio G(\R)$ is surjective.
It follows that the homomorphism $i_*$ in the sequence \eqref{e:rho-i} is surjective.
We conclude  that the sequence \eqref{e:rho-i} is exact,
which completes the proof of the theorem.
\end{proof}

\noindent{\sc Acknowledgements.}
This note was conceived when the first-named author
was a guest of the Laboratory of Mathematics of Orsay
(Laboratoire de Math\'ematiques d’Orsay, Universit\'e Paris-Saclay).
He is grateful to the Laboratory for hospitality and good working conditions.

\end{document}